\documentclass[10pt]{amsart}

\usepackage{amssymb,amsthm,amstext,amsfonts,amsmath}
\usepackage[spanish,english]{babel}

\newtheoremstyle{theorem}
     {11pt}
     {11pt}
     {}
     {}
     {\bfseries}
     {}
     {.5em}
     {\noindent\thmnumber{#2}. \thmname{#1}\thmnote{#3}}

\theoremstyle{theorem}

\newtheorem{thm}{Theorem}[section]
\newtheorem{lemma}[thm]{Lemma}
\newtheorem{propo}[thm]{Proposition}

\newtheorem{ques}[thm]{Question}

\newcommand{\F}{\mathcal{F}}
\newcommand{\css}{\mathcal{P}(\omega)}
\newcommand{\p}{\mathcal{P}}
\newcommand{\csf}{{}\sp{\omega}2}
\newcommand{\I}{\mathcal{I}}
\newcommand{\U}{\mathcal{U}}
\newcommand{\G}{\mathcal{G}}
\newcommand{\B}{\mathcal{B}}

\title{Non-meager $P$-filters are Countable Dense Homogeneous}

\author[Hern\'andez-Guti\'errez]{Rodrigo Hern\'andez-Guti\'errez}
\author[Hru\v s\'ak]{Michael Hru\v s\'ak}

\address{Centro de Ciencias Matem\'aticas, UNAM, A.P. 61-3, Xangari, Morelia, Michoac\'an, 58089, M\'exico}	
	\email[Hern\'andez-Guti\'errez]{rod@matmor.unam.mx}
	\email[Hru\v s\'ak]{michael@matmor.unam.mx}

\thanks{This paper is part of the first listed author's doctoral dissertation. Research was supported by CONACyT scholarship for Doctoral Students. 
The second listed author acknowledges support from PAPIIT grant IN102311 and CONACyT grant 177758.}

\date{March 1, 2013, revised June 16, 2017}

\begin{document}

\begin{abstract}
We prove that if $\F$ is a non-meager $P$-filter, then both $\F$ and ${}\sp\omega\F$ are countable dense homogeneous spaces.
\end{abstract}

\maketitle

\section{Introduction}

{\it All spaces considered are separable and metrizable.}

A separable space $X$ is \emph{countable dense homogeneous} (\emph{CDH} for short) if whenever $D$ and $E$ are countable dense subsets of $X$, 
there exists a homeomorphism $h:X\to X$ such that $h[D]=E$. Using the now well-known back-and-forth argument, Cantor \cite{cantor} gave the 
first example of a CDH space: the real line. In fact, many other important spaces are CDH, e.g. the Euclidean spaces, the Hilbert cube and the 
Cantor set. Results from \cite{bennett} and \cite{ungar} provide general classes of CDH spaces that include the examples mentioned.
 In \cite{cdhdefinable} and \cite{vm} the reader can find summaries of past research and bibliography about CDH spaces.

In \cite{openprob}, Fitzpatrick and Zhou posed the following problems.

\begin{ques}\label{quesop1}
Does there exist a CDH metrizable space that is not completely metrizable?
\end{ques}

\begin{ques}\label{quesop2}
For which $0$-dimensional subsets $X$ of $\mathbb{R}$ is ${}\sp\omega X$ CDH?
\end{ques}

 Concerning these two problems, the following results have been obtained.

\begin{thm}\cite{cdhdefinable}\label{cdhdef}
Let $X$ be a separable metrizable space.
\begin{itemize}
\item If $X$ is CDH and Borel, then $X$ is completely metrizable.
\item If ${}\sp\omega X$ is CDH, then $X$ is a Baire space.
\end{itemize}
\end{thm}

\begin{thm}\cite{lambda}\label{cdhlambda}
There is a CDH set of reals $X$ of size $\omega_1$ that is a $\lambda$-set\footnote{Recall that a set of reals $X$ is a \emph{$\lambda$-set} if every countable 
subset of $X$ is a relative $G_\delta$ set.} and thus, not completely metrizable.
\end{thm}

The techniques used in the proof of Theorem \ref{cdhlambda} produce spaces that are not Baire spaces, so by Theorem \ref{cdhdef} they 
cannot answer Question \ref{quesop2}.

There is a natural bijection between the Cantor set $\csf$ and $\css$ via characteristic functions. In this way we may identify $\css$ with the Cantor set. 
Thus, any subset of $\css$ can be thought of as a separable metrizable space.

Recall that a set $\F\subset\css$  is a \emph{filter} (on $\omega$) if (a) $X\in\F$ and $\emptyset\notin\F$; (b) if $A,B\in\F$, then $A\cap B\in\F$ and 
(c) if $A\in\F$ and $A\subset B\subset\omega$, then $B\in\F$. For $\mathcal{X}\subset\css$, let $\mathcal{X}\sp\ast=\{x\subset\omega:\omega\setminus 
x\in\mathcal{X}\}$. Then a set $\I\subset\css$ is an \emph{ideal} if and only if $\I\sp\ast$ is a filter. We will assume that all filters contain the 
Fr\'echet filter $\{x\subset\omega:\omega\setminus x\textrm{ is finite}\}$ (dually, all ideals contain the set of finite subsets of $\omega$). 
An \emph{ultrafilter} is a maximal filter with respect to inclusion. 

If $A,B$ are sets, $A\subset\sp\ast B$ means that $A\setminus B$ is finite. A filter $\F$ is called a \emph{$P$-filter} if given 
$\{X_n:n<\omega\}\subset\F$ there exists $X\in\F$ such that $X\subset\sp\ast X_n$ for all $n<\omega$. Such an $X$ is called a 
\emph{pseudo-intersection} of $\{X_n:n<\omega\}$. Dually, $\I$ is a \emph{$P$-ideal} if $\I\sp\ast$ is a $P$-filter. An ultrafilter that is a 
$P$-filter is called a \emph{$P$-point}.

Considering ultrafilters as topological spaces, the following results were obtained recently by Medini and Milovich.

\begin{thm}\cite[Theorems 15, 21, 24, 41, 43 and 44]{medinimilovich}\label{ultrathm}
Assume $MA(countable)$. Then there are ultrafilters $\U\subset\css$ with any of the following properties: $(a)$ $\U$ is CDH and a $P$-point, $(b)$ $\U$ is CDH and not a $P$-point, $(c)$ $\U$ is not CDH and not a $P$-point, and $(d)$ ${}\sp\omega\U$ is CDH.
\end{thm}

Since ultrafilters do not even have the Baire property (\cite[4.1.1]{bart}), Theorem \ref{ultrathm} gives a consistent answer to Question \ref{quesop1} and a consistent example for Question \ref{quesop2}.

The purpose of this note is to extend these results on ultrafilters to a wider class of filters on $\omega$. In particular, we prove the following result, 
which answers Questions 3, 5 and 11 of \cite{medinimilovich}.

\begin{thm}\label{bigthm}
Let $\F$ be a non-meager $P$-filter on $\css$ extending the Fr\'echet filter. Then both $\F$ and ${}\sp\omega\F$ are CDH.
\end{thm}

It is known that non-meager filters do not have the Baire property (\cite[4.1.1]{bart}). However, the existence of non-meager $P$-filters is an open question 
(in ZFC). It is known that the existence of non-meager $P$-filters follows from $\mathbf{cof}([\mathfrak{d}]\sp\omega)=\mathfrak{d}$ 
(where $\mathfrak{d}$ is the dominating number, see \cite[1.3.A]{bart}).  Hence, if all $P$-filters are meager then there is an inner model with 
large cardinals. See \cite[4.4.C]{bart} or \cite{just-et-all} for a detailed description of this problem. 

Note that every CDH filter has to be non-definable in the following sense.

\begin{propo}\label{mustbenonmeager}
Let $\F$ be a filter on $\css$ extending the Fr\'echet filter. If one of $\F$ or ${}\sp\omega\F$ is CDH, then $\F$ is non-meager.
\end{propo}
\begin{proof}
If ${}\sp\omega{\F}$ is CDH, then $\F$ is non-meager by \cite[Theorem 3.1]{cdhdefinable}. Assume that $\F$ is CDH. If $\F$ is the Fr\'echet filter, 
then $\F$ is countable, hence not CDH. If $\F$ is not the Fr\'echet filter, there exists $x\in\F$ such that $\omega\setminus x$ is infinite. 
Thus, $C=\{y:x\subset y\subset\omega\}$ is a copy of the Cantor set contained in $\F$.  

If $\F$ were meager, we arrive at a contradiction as follows: Let $D\subset \F$ be a countable dense subset of $\F$ such that $D\cap C$ is dense in $C$. 
Since $\F$ is meager in itself, by \cite[Lemma 2.1]{cdhdefinable}, there is a countable dense subset $E$ of $\F$ that is a $G_\delta$ set relative to $\F$. 
Let $h:\F\to \F$ be a homeomorphism such that $h[D]=E$. Then $h[D\cap C]$ is a countable dense subset of the Cantor set $h[C]$ that is a relative $G_\delta$ 
subset of $h[C]$, which is impossible. So $\F$ is non-meager and the proof is complete.
\end{proof}

Notice that $(\css,\bigtriangleup,\emptyset)$ is a topological group (where $A\bigtriangleup B$ denotes the symmetric difference of $A$ and $B$) 
as it corresponds to addition modulo $2$ in $\csf$. Given a filter $\F\subset\css$, the dual ideal $\F\sp\ast$ is homeomorphic to $\F$ by means of 
the map that sends each subset of $\omega$ to its complement. Notice that $\emptyset\in\F\sp\ast$ and $\F\sp\ast$ is closed under $\bigtriangleup$. 
Moreover, for each $x\in\css$, the function $y\mapsto y\bigtriangleup x$ is a autohomeomorphism of $\css$. From this it is easy to see that $\F$ is 
homogeneous. Thus, by \cite[Proposition 3]{medinimilovich}, ``non-meager'' in Proposition \ref{mustbenonmeager} can be replaced by ``Baire space''.

In \cite[Theorem 1.2]{marciszewski} it is proved that a filter $\F$ is hereditarily Baire if and only if $\F$ is a non-meager $P$-filter. 
By Theorem \ref{ultrathm}, it is consistent that not all CDH ultrafilters are $P$-points so it is consistent that there are CDH filters 
that are not hereditarily Baire. These observations answer Question 4 in \cite{medinimilovich}. 

Recall that Theorem \ref{ultrathm} also shows that it is consistent that there exist non-CDH non-meager filters.

\begin{ques}
Is there a combinatorial characterization of CDH filters?
\end{ques}

\begin{ques}
Is there a CDH filter (ultrafilter) in ZFC? Is there a non-CDH and non-meager filter (ultrafilter) in ZFC?
\end{ques}

\section{Proof of Theorem \ref{bigthm}}

For any set $X$, let $[X]\sp{<\omega}$ and $[X]\sp\omega$ denote the set of its finite and countably infinite subsets, respectively. 
Also ${}\sp{<\omega}X=\bigcup\{{}\sp{n}X:n<\omega\}$.

Since any filter is homeomorphic to its dual ideal, we may alternatively speak about a filter or its dual ideal. In particular, the 
following result is better expressed in the language of ideals. Its proof follows from \cite[Lemma 20]{medinimilovich}, we include 
it for the sake of completeness.

\begin{lemma}\label{restriction}
Let $\I\subset\css$ be an ideal, $f:\css\to\css$ a continuous function and $D$ a countable dense subset of $\I$. If there exists $x\in\I$ 
such that $\{d\bigtriangleup f(d):d\in D\}\subset\mathcal{P}(x)$, then $f[\I]=\I$.
\end{lemma}
\begin{proof}
Since $D$ is dense in $\css$ and $d\bigtriangleup f(d)\subset x$ for all $d\in D$, by continuity it follows that $y\bigtriangleup f(y)\subset x$ 
for all $y\in\css$. Then $y\bigtriangleup f(y)\in\I$ for all $y\in\css$. Since $\I$ is closed under $\bigtriangleup$ and $a\bigtriangleup a=\emptyset$ for all $a\in\css$, it is easy to see that $y\in\I$ if and only if $f(y)\in\I$ for all $y\in\css$.
\end{proof}

Let $\mathcal{X}\subset[\omega]\sp\omega$. A tree $T\subset{}\sp{<\omega}([\omega]\sp{<\omega})$ is called a \emph{$\mathcal{X}$-tree of finite sets} if for each $s\in T$ there is $X_s\in\mathcal{X}$ such that for every $a\in[X_s]\sp{<\omega}$ we have $s\sp\frown a\in T$. It turns out that non-meager $P$-filters have a very useful combinatorial characterization as follows.

\begin{lemma}\cite[Lemma 1.3]{laflamme}\label{treelemma}
Let $\F$ be a filter on $\css$ that extends the Fr\'echet filter. Then $\F$ is a non-meager $P$-filter if and only if every $\F$-tree of finite sets has a branch whose union is in $\F$.
\end{lemma}

Next we prove a combinatorial property that will allow us to construct autohomeomorphisms of the Cantor set that restrict to a given ideal. 
For $x\in\css$, let $\chi(x)\in\csf$ be its characteristic function.

\begin{lemma}\label{combinatorics}
Let $\I$ be a non-meager $P$-ideal and $D_0$, $D_1$ be two countable dense subsets of $\I$. Then there exists $x\in\I$ such that
\begin{itemize}
\item[(i)] for each $d\in D_0\cup D_1$, $d\subset\sp\ast x$ and
\item[(ii)] for each $i\in 2$, $d\in D_i$, $n<\omega$ and $t\in{}\sp{n\cap x}2$, there exists $e\in D_{1-i}$ such that $d\setminus x=e\setminus x$ and $\chi(e)\restriction_{n\cap x}=t$.
\end{itemize}
\end{lemma}
\begin{proof}
Let $\F=\I\sp\ast$. We will construct an $\F$-tree of finite sets $T$ and use Lemma \ref{treelemma} to find $x\in\I$ with the properties listed. Let us give an enumeration $(D_0\cup D_1)\times{{}\sp{<\omega}2}=\{(d_n,t_n):n<\omega\}$ such that $\{d_n:n\equiv i\ ({mod}\ 2)\}=D_i$ for $i\in 2$.

The definition of $T$ will be by recursion. For each $s\in T$ we also define $n(s)<\omega$, $F_s\in\F$ and $\phi_s:{dom}(s)\to D_0\cup D_1$ so that the following properties hold.

\begin{itemize}
\item[(1)] $\forall s,t\in T\ (s\subsetneq t\Rightarrow n(s)<n(t))$,
\item[(2)] $\forall s\in T\ \forall k<{dom}(s)\ (s(k)\subset n(s\restriction_{k+1})\setminus n(s\restriction_k))$,
\item[(3)] $\forall s,t\in T\ (s\subset t\Rightarrow F_t\subset F_s)$,
\item[(4)] $\forall s\in T\ (F_s\subset\omega\setminus n(s))$,
\item[(5)] $\forall s,t\in T\ (s\subset t\Rightarrow \phi_s\subset\phi_t)$,
\item[(6)] $\forall s\in T,\textrm{ if }k={dom}(s)\ ((d_{k-1}\cup\phi_s(k-1))\setminus n(s)\subset\omega\setminus F_s)$.
\end{itemize}

Since $\emptyset\in T$, let $n(\emptyset)=0$ and $F_\emptyset=\omega$. Assume we have $s\in T$ and $a\in [F_s]\sp{<\omega}$, we have to define everything for $s{}\sp\frown a$. Let $k={dom}(s)$. We start by defining $n(s{}\sp\frown a)=\max{\{k,\max{(a)},{dom}(t_k)\}}+1$. Next we define $\phi_{s{}\sp\frown a}$. We only have to do it at $k$ because of (5). We have two cases.

\underline{Case 1}. There exists $m<{dom}(t_k)$ with $t_k(m)=1$ and $m\in s(0)\cup\ldots\cup s(k-1)$. We simply declare $\phi_{s{}\sp\frown a}(k)=d_k$.

\underline{Case 2}. Not Case 1. We define $r_{s{}\sp\frown a}\in{}\sp{n(s{}\sp\frown a)}2$ in the following way.

$$
r_{s{}\sp\frown a}(m)=\left\{
\begin{tabular}{ll}
$d_k(m)$,& if $m\in s(0)\cup\ldots\cup s(k-1)\cup a$,\\
$t_k(m)$,& if $m\in{dom}(t_k)\setminus(s(0)\cup\ldots\cup s(k-1)\cup a)$,\\
$1$,& in any other case.
\end{tabular}
\right.
$$

Let $i\in 2$ be such that $i\equiv k\ ({mod}\ 2)$. So $d_k\in D_i$, let $\phi_{s{}\sp\frown a}(k)\in D_{1-i}$ be such that $\phi_{s{}\sp\frown a}(k)\cap n(s{}\sp\frown a)=(r_{s{}\sp\frown a})\sp{-1}(1)$, this is possible because $D_{1-i}$ is dense in $\css$. Finaly, define
$$
F_{s{}\sp\frown a}=(F_s\cap(\omega\setminus d_{k-1})\cap(\omega\setminus\phi_{s{}\sp\frown a}(k-1)))\setminus n(s{}\sp\frown a).
$$

Clearly, $F_{s{}\sp\frown a}\in\F$ and it is easy to see that conditions (1) -- (6) hold.

By Lemma \ref{treelemma}, there exists a branch $\{(y_0,\ldots,y_n):n<\omega\}$ of $T$ whose union $y=\bigcup\{y_n:n<\omega\}$ is in $\F$. Let $x=\omega\setminus y\in\I$. We prove that $x$ is the element we were looking for. It is easy to prove that (6) implies (i). 

We next prove that (ii) holds. Let $i\in 2$, $n<\omega$, $t\in{}\sp{n\cap x}2$ and $d\in D_i$. Let $k<\omega$ be such that $(d_k,t_k)=(d,t\sp\prime)$, where $t\sp\prime\in{}\sp{n}2$ is such that $t\sp\prime\restriction_{n\cap x}=t$ and $t\sp\prime\restriction_{n-x}=\underline{0}$. Consider step $k+1$ in the construction of $y$, that is, the step when $y_{k+1}$ was defined. Notice that we are in Case 2 of the construction and $r_{y_{k+1}}$ is defined. Then $\phi_{y_{k+1}}(k)=e$ is an element of $D_{1-i}$. It is not very hard to see that $d\setminus x=e\setminus x$ and $\chi(e)\restriction_{n\cap x}=t$. This completes the proof of the Lemma.
\end{proof}

We will now show that it is enough to prove Theorem \ref{bigthm} for $\F$. Recall the following characterization of non-meager filters.

\begin{lemma}\cite[Theorem 4.1.2]{bart}\label{lemmanonmeager}
Let $\F$ be a filter. Then $\F$ is non-meager if and only if for every partition of $\omega$ into finite sets $\{J_n:n<\omega\}$, there is $X\in\F$ such that $\{n<\omega: X\cap J_n=\emptyset\}$ is infinite.
\end{lemma}

The following was originally proved by Shelah (see \cite[Fact 4.3, p. 327]{shelah-proper}). We include a proof of this fact for the convenience of the reader.

\begin{lemma}\label{omegapower}
If $\F$ is a non-meager $P$-filter, then ${}\sp\omega\F$ is homeomorphic to a non-meager $P$-filter.
\end{lemma}
\begin{proof}
Let 
$$
\G=\{A\subset\omega\times\omega:\forall n<\omega (A\cap(\{n\}\times\omega)\in\F)\}.
$$
Notice that $\G$ is homeomorphic to ${}\sp\omega{\F}$. It is easy to see that $G$ is a filter on $\omega\times\omega$. We next prove that $\G$ is a non-meager $P$-filter. 

Let $\{A_k:k<\omega\}\subset\G$. For each $\{k,n\}\subset\omega$, let $A_k\sp n=\{x\in\omega:(n,x)\in A_k\}\in\F$. Since $\F$ is a $P$-filter, there is $A\in\F$ such that $A\subset\sp\ast A_k\sp n$ for all $\{k,n\}\subset\omega$. Let $f:\omega\to\omega$ be such that $A\setminus f(n)\subset A_k\sp n$ for all $k\leq n$. Let 
$$
B=\bigcup\big\{\{n\}\times(A\setminus f(n)):n<\omega\big\}.
$$
Then it is easy to see that $B\in\G$ and $B$ is a pseudointersection of $\{A_n:n<\omega\}$. So $\G$ is a $P$-filter.

Let $\{J_k:k<\omega\}$ a partition of $\omega\times\omega$ into finite subsets. Recursively, we define a sequence $\{F_n:n<\omega\}\subset\F$ and a sequence $\{A_n:n<\omega\}\subset[\omega]\sp{\omega}$ such that $A_{n+1}\subset A_n$ and $A_n\subset\{k<\omega:J_k\cap(\{n\}\times F_n)=\emptyset\}$ for all $n<\omega$. 

For $n=0$, since $\F$ is non-meager, by Lemma \ref{lemmanonmeager} there is $F_0\in\F$ such that $\{k<\omega:J_k\cap(\{0\}\times F_0)=\emptyset\}$ is infinite, call this last set $A_0$. Assume that we have the construction up to $m<\omega$, then $\B=\{J_k\cap (\{m+1\}\times \omega):k\in A_m\}$ is a family of pairwise disjoint finite subsets of $\{m+1\}\times\omega$. If $\bigcup\B$ is finite, let $F_{m+1}\in\F$ be such that $F_{m+1}\cap(\bigcup\B)=\emptyset$ and let $A_{m+1}=A_m$. If $\bigcup\B$ is infinite, let $\{B_k:k\in A_m\}$ be any partition of $(\{m+1\}\times\omega)\setminus\bigcup\B$ into finite subsets (some possibly empty). For each $k\in A_m$, let $C_k=(J_k\cap(\{m+1\}\times\omega))\cup B_k$. Then $\{C_k: k\in A_m\}$ is a partition of $\{m+1\}\times\omega$ into finite sets so by Lemma \ref{lemmanonmeager}, there is $F_{m+1}\in\F$ such that $\{k\in A_m: C_k\cap(\{m+1\}\times F_{m+1})=\emptyset\}$ is infinite, call this set $A_{m+1}$. This completes the recursion.

Define an increasing function $s:\omega\to\omega$ such that $s(0)=\min{A_0}$ and $s(k+1)=\min{(A_{k+1}\setminus\{s(0),\ldots,s(k)\})}$ for $k<\omega$. Also, define $t:\omega\to\omega$ such that $t(0)=0$ and $t(n+1)=\min{\{m<\omega:(J_{s(0)}\cup\ldots\cup J_{s(n)})\cap(\{n+1\}\times\omega)\subset \{n+1\}\times m\}}$. Finally, let
$$
G=\bigcup{\{\{n\}\times(F_n\setminus t(n)):n<\omega\}}.
$$

Then $G\in\G$ and for all $k<\omega$, $G\cap J_{s(k)}=\emptyset$. Thus, $\G$ is non-meager by Lemma \ref{lemmanonmeager}.
\end{proof}

We now have everything ready to prove our result.

\begin{proof}[{\bf Proof of Theorem \ref{bigthm}}]
By Lemma \ref{omegapower}, it is enough to prove that $\F$ is CDH, equivalently that $\I=\F\sp\ast$ is CDH. Let $D_0$ and $D_1$ be two countable dense subsets of $\I$ and let $x\in\I$ be given by Lemma \ref{combinatorics}.

We will construct a homeomorphism $h:\css\to\css$ such that $h[D_0]=D_1$ and 
\begin{equation*}\tag{$\star$}
\forall d\in D\ (d\bigtriangleup h(d)\subset x).
\end{equation*}

By Lemma \ref{restriction}, $h[\I]=\I$ and we will have finished.

We shall define $h$ by approximations. By this we mean the following. We will give a strictly increasing sequence $\{n(k):k<\omega\}\subset\omega$ and in step $k<\omega$ a homeomorphism (permutation) $h_k:\p(n(k))\to\p(n(k))$ such that
\begin{equation*}\tag{$\ast$}
\forall j<k<\omega\ \forall a\in\p(n(k))\ (h_k(a)\cap n(j)=h_j(a\cap n(j))).
\end{equation*}

By $(\ast)$, we can define $h:\css\to\css$ to be the inverse limit of $\{h_k:k<\omega\}$, which is a homeomorphism.

Let $D_0\cup D_1=\{d_n:n<\omega\}$ in such a way that $\{d_n:n\equiv i\ ({mod}\ 2)\}=D_i$ for $i\in 2$. To make sure that $h[D_0]=D_1$, in step $k$ we have to decide the value of $h(d_k)$ when $k$ is even and the value of $h\sp{-1}(d_k)$ when $k$ is odd. We do this by approximating a bijection $\pi:D_0\to D_1$ in $\omega$ steps by a chain of finite bijections $\{\pi_k:k<\omega\}$ and letting $\pi=\bigcup\{\pi_k:k<\omega\}$. In step $k<\omega$, we would like to have $\pi_k$ defined on some finite set so that the following conditions hold whenever $\pi_k\subset\pi$:
\begin{itemize}
\item[$(a)_k$] if $j<k$ is even, then $h_k(d_j\cap n(k))=\pi(d_j)\cap {n(k)}$, and
\item[$(b)_k$] if $j<k$ is odd, then $h_k(d_j\cap n(k))=\pi\sp{-1}(d_j)\cap n(k)$.
\end{itemize}

Notice that once $\pi$ is completely defined, if $(a)_k$ and $(b)_k$ hold for all $k<\omega$, then $h[D]=E$. As we do the construction, we need to make sure that the following two conditions hold.
\begin{align*}
(c)_k & \ \ \ \forall i\in n(k)\setminus x\ \ \forall a\in\p(n(k))\ (i\in a\Leftrightarrow i\in h_k(a))\\
(d)_k & \ \ \ \forall d\in{dom}(\pi_k)\ (d\setminus x=\pi_k(d)\setminus x)
\end{align*}

Condition $(c)_k$ is a technical condition that will help us carry out the recursion. Notice that if we have condition $(d)_k$ for all $k<\omega$, then $(\star)$ will hold.

Assume that we have defined $n(0)<\ldots<n(s-1)$, $h_0,\ldots, h_{s-1}$ and a finite bijection $\pi_s\subset D_0\times D_1$ with $\{d_r:r<s\}\subset{dom}(\pi_s)\cup{dom}(\pi_s\sp{-1})$ in such a way that if $\pi\supset\pi_s$, then $(a)_{s-1}$, $(b)_{s-1}$, $(c)_{s-1}$ and $(d)_{s-1}$ hold. Let us consider the case when $s$ is even, the other case can be treated in a similar fashion.

If $d_s=\pi_s\sp{-1}(d_r)$ for some odd $r<s$, let $n(s)=n(s-1)+1$. If we let $\pi_{s+1}=\pi_s$, it is easy to define $h_s$ so that it is compatible with $h_{s-1}$ in the sense of $(\ast)$, in such a way that $(a)_s$, $(b)_s$, $(c)_s$ and $(d)_s$ hold for any $\pi\supset\pi_{s+1}$. So we may assume this is not the case.

Notice that the set $S=\{d_r:r<s+1\}\cup\{\pi_s(d_r):r<s,r\equiv 0\ (mod\ 2)\}\cup\{\pi_s\sp{-1}(d_r):r<s,r\equiv 1\ (mod\ 2)\}$ is finite. Choose $p<\omega$ so that $d_s\setminus p\subset x$. Let $r_0=h_{s-1}(d_s\cap n(s-1))\in\p(n(s-1))$. Choose $n(s-1)<m<\omega$ and $t\in{}\sp{m\cap x}2$ in such a way that $t\sp{-1}(1)\cap n(s-1)=r_0\cap n(s-1)\cap x$ and $t$ is not extended by any element of $\{\chi(a):a\in S\}$. By Lemma \ref{combinatorics}, there exists $e\in E$ such that $d_s\setminus x=e\setminus x$ and $\chi(e)\restriction_{m\cap x}=t$. Notice that $e\notin S$ and $\chi(e)\restriction_{n(s-1)}=r_0$ by condition $(c)_{s-1}$. We define $\pi_{s+1}=\pi_s\cup\{(d_s,e)\}$. Notice that $(d)_s$ holds in this way.

Now that we have decided where $\pi$ will send $d_s$, let $n(s)>\max{\{p,m\}}$ be such that there are no two distinct $a,b\in S\cup\{\pi_{s+1}(d_s)\}$ with $a\cap n(s)=b\cap n(s)$. Topologically, all elements of $S\cup\{\pi_{s+1}(d_s)\}$ are contained in distinct basic open sets of measure $1/(n(s)+1)$.

Finally, we define the bijection $h_s:\p(n(s))\to\p(n(s))$. For this part of the proof we will use characteristic functions instead of subsets of $\omega$ (otherwise the notation would become cumbersome). Therefore, we may say $h_{r}:{}\sp{n(r)}2\to{}\sp{n(r)}2$ is a homeomorphism for $r<s$. 

Let $(q,q\sp\prime)\in{}\sp{n(s-1)}2\times{}\sp{n(s)\setminus x}2$ be a pair of compatible functions. Notice that $(h_{s-1}(q),q\sp\prime)$ are also compatible by $(c)_{s-1}$. Consider the following condition.
$$
\triangledown(q,q\sp\prime):\ \ \ 
\forall a\in{}\sp{n(s)}2\ (q\cup q\sp\prime\subset a\Leftrightarrow h_{s-1}(q)\cup q\sp\prime\subset h_s(a))
$$
Notice that if we define $h_s$ so that $\triangledown(q,q\sp\prime)$ holds for each pair $(q,q\sp\prime)\in{}\sp{n(s-1)}2\times{}\sp{n(s)\setminus x}2$ of compatible functions, then $(\ast)$ and $(c)_s$ hold as well.

Then for each pair $(q,q\sp\prime)\in{}\sp{n(s-1)}2\times{}\sp{n(s)\setminus x}2$ of compatible functions we only have to find a bijection $g:{}\sp{T}2\to{}\sp{T}2$, where $T=(n(s)\cap x)\setminus n(s-1)$ (this bijection will depend on such pair) and define $h_s:{}\sp{n(s)}2\to{}\sp{n(s)}2$ as 
$$
h_s(a)=h_{s-1}(q)\cup q\sp\prime\cup g(f\restriction_T),
$$
whenever $a\in{}\sp{n(s)}2$ and $q\cup q\sp\prime\subset a$. There is only one restriction in the definition of $g$ and it is imposed by conditions $(a)_s$ and $(b)_s$; namely that $g$ is compatible with the bijection $\pi_{s+1}$ already defined. However by the choice of $n(s)$ this is not hard to do. This finishes the inductive step and the proof.

\end{proof}

\end{document}